\documentclass[10pt]{amsart}
\usepackage{amsfonts,amssymb,amsmath,amsthm}
\usepackage{enumerate}
\usepackage{mathrsfs}

\newcommand{\rem}[1]{}

\newtheorem{theorem}{Theorem}[section]

\newtheorem{lemma}[theorem]{Lemma}
\newtheorem{proposition}[theorem]{Proposition}
\theoremstyle{definition}
\newtheorem{definition}[theorem]{Definition}

\newtheorem{example}[theorem]{Example}
\newtheorem{remark}[theorem]{Remark}

\numberwithin{equation}{section}

\def\NN{\mathbb N}
\def\RR{\mathbb R}
\newcommand\cL{{\mathcal{L}}}
\newcommand{\rL}{{{\mathrm{L}}}}
\newcommand{\rC}{{{\mathrm{C}}}}
\newcommand{\rT}{{{\mathrm{T}}}}
\renewcommand{\sb}{{\mathop{{\rm s}}}}

\def\dd{\mathrm{d}}
\def\dd{\di}

\newcommand\sse{\subseteq}

\newcommand\tmo{^{-1}}
\newcommand\mo{_{-1}}
\newcommand\spr{\operatorname{spr}}
\newcommand\textimplies{{\mathsurround=0pt\ensuremath{\ \Rightarrow}\ }}
\newcommand\timplies{\textimplies}

\newcommand\dom{\operatorname{dom}}
\newcommand\la{\lambda}
\newcommand\partof{{\mskip1mu\vrule height.5ex depth.9ex\mskip2mu}}
\renewcommand\phi{\varphi}
\newcommand\Emop{E_{-1,+}}
\newcommand\eul{{\rm e}}
\newcommand\di{\mathclose{}\,\mathrm{d}}

\renewcommand\le{\leqslant}
\renewcommand\leq{\leqslant}
\renewcommand\ge{\geqslant}
\renewcommand\geq{\geqslant}

\newenvironment{iiv}{\begin{enumerate}[(i)]}{\end{enumerate}}

\title[Perturbations of positive semigroups]{Perturbations of positive semigroups on AM-spaces}

\author[A. B\'atkai]{Andr\'as B\'atkai}
\author[B. Jacob]{Birgit Jacob}
\author[J. Voigt]{J\"urgen Voigt}
\author[J. Wintermayr]{Jens Wintermayr}

\address{A.B., B.J., J.W., Bergische Universit\"at Wuppertal, School of Mathematics and Natural Sciences, Gau\ss  strasse 20, 42119, Wuppertal, Germany}
\address{A.B., P\"adagogische Hochschule Vorarlberg, Liechtensteinerstra\ss e 33 - 37, 6800 Feldkirch, Austria}
\address{J.V., Technische Universit\"at Dresden, Fachrichtung Mathematik, 01062
Dresden, Germany}

\email{andras.batkai@ph-vorarlberg.ac.at}
\email{jacob@math.uni-wuppertal.de}
\email{juergen.voigt@tu-dresden.de}
\email{wintermayr@uni-wuppertal.de}

\begin{document}

\begin{abstract}We consider positive perturbations of positive semigroups on 
AM-spaces and prove a result which is the dual counterpart of a famous 
perturbation result of Desch in AL-spaces. As an application we present
unbounded perturbations of the shift semigroup.
\end{abstract}

\maketitle

\section{Introduction}\label{Intro}

Strongly continuous semigroups play a central role in operator theory, partial
differential equations, and linear systems theory, as documented in the
monographs by Engel and Nagel \cite{EN}, Pazy \cite{Pazy}, Davies \cite{Davies},
Goldstein \cite{Gol:85}, Tucsnak and Weiss \cite{TW}, or Jacob and Zwart
\cite{MR2952349}.

One of the central problems of operator semigroup theory is to decide whether a
concrete operator is the generator of a semigroup and how this semigroup is
represented. Though the famous Hille--Yosida theorem provides a complete
characterization of semigroup generators, it is practically never used in
applications because of the difficult technical conditions appearing there.

One idea is to write complicated operators as the sum of simple ones and this
is why perturbation theory became one of the major topics in semigroup theory.
The main question is: Supposing $A$ generates a $C_0$-semigroup $(T(t))_{t\geq
0}$, under which conditions on $B$ does the operator $A+B$ (suitably defined)
generate a $C_0$-semigroup?

There has been enormous development in the perturbation theory of operator
semigroups, which is well documented in the monographs by Engel and Nagel
\cite[Chapter III]{EN}, Kato \cite{Kat:80}, Banasiak and Arlotti \cite{BA},
Tucsnak and Weiss \cite[Section 5.4]{TW}, B\'atkai, Kramar Fijav\v{z} and Rhandi
\cite[Chapter 13]{BKR}.

In this note we concentrate on perturbations of positive semigroups in Banach
lattices. For Banach lattices and positive operators on them we refer to the
monographs by Aliprantis and Burkinshaw \cite{AB} or Schaefer \cite{Sch}, and
for positive semigroups to Nagel \cite{N} and B\'atkai, Kramar Fijav\v{z} and
Rhandi \cite{BKR}.

Our work is motivated by a well-known perturbation result, originally due to
Desch \cite{D} and Voigt \cite{Voi89}, which we cite here. To be able to
formulate it, we need some notions from Banach lattices and positive operators,
which will be explained in the next section.

\begin{theorem}\label{cor:miya}
Let $A$ be the generator of a positive $C_0$-semigroup $(T(t))_{t\geq 0}$ on a
real AL-space $E$. Let $B\colon \dom(A)\to E$ be a positive operator, and
assume that there exists $\la > \sb(A)$ such that $\spr(B(\la - A)\tmo) < 1$
(where $\spr$ denotes the spectral radius).
Then $A+B$ generates a positive $C_0$-semigroup on $E$.
\end{theorem}


We refer to Remark~\ref{rem-main-result}(b) for this formulation of Desch's
result.

Let us formulate here the main theoretical result of our paper.

\begin{theorem}\label{AMresult}
Let $A$ be the generator of a positive $C_0$-semigroup $(T(t))_{t\geq 0}$ on a
real AM-space $E$. Let $B\colon E \to E_{-1}$ be a positive operator and
suppose that there is a $\lambda>\sb(A)$ such that
$\spr\bigl((\lambda-A_{-1})^{-1}B\bigr)<1$. Then the part $(A_{-1}+B)\partof _E$
of $A_{-1}+B$ in $E$ generates a positive $C_0$-semigroup on $E$.
\end{theorem}

We recall that the part of $A_{-1}+B$ in $E$ is the restriction of $A_{-1}+B$
to the domain $\dom\bigl((A_{-1}+B)\partof_E\bigr):=\{f\in E: (A_{-1}+B)f\in
E\}$, considered as an operator in $E$.
The extrapolation space $E_{-1}$ and the extrapolated operator $A_{-1}$ are
explained in Section \ref{posextraspaces}. In Section \ref{pert-res-pos} we discuss a technical tool, the perturbation of so-called resolvent positive operators. The main result is proved in Section
\ref{perturbationtheory}. Finally, an application is discussed in detail in
Section \ref{examples}.

\section{Extrapolation Spaces and Positivity}\label{posextraspaces}

Let $X$ be a Banach space and $(T(t))_{t\geq 0}$
a $C_0$-semigroup on $X$, with generator $A$. For $\lambda\in\rho(A)$ we
define
the extrapolation space as the completion
$X_{-1}:=(X,\|\cdot\|_{-1})^{\thicksim}$, where $\|f\|_{-1}:=\|R(\lambda,A)f\|$.
Here $R(\lambda,A) = (\lambda-A)^{-1}$ is the resolvent of $A$.
For all $t\ge0$, the operator $T(t)$ has a unique extension 
$T_{-1}(t)\in\cL(X\mo)$, and $(T_{-1}(t))_{t\geq 0}$ is a $C_0$-semigroup on 
$X\mo$, with generator
$A_{-1}$ satisfying $\dom(A_{-1})=X$. 
Moreover, these definitions are
independent of the choice of $\lambda\in\rho(A)$, meaning that a different
$\lambda\in\rho(A)$ generates the same space with equivalent norms. We refer
the reader to Engel and Nagel \cite[Chapter II.5]{EN} for these properties 
and for more on this subject.

\begin{definition}\label{pos}
Let $E$ be a Banach lattice and $E_{-1}$ the extrapolation space for the
positive $C_0$-semigroup $(T(t))_{t\geq 0}$.
We say that $f\in E_{-1}$ is \em positive\em, if $f$ belongs to the closure of
$E_+$ in $E_{-1}$.
We denote by $\Emop$ the set of all positive elements in $E_{-1}$.
\end{definition}
From the definition the set of positive elements satisfies $E_+\sse \Emop$.
By
$$\sb(A)=\sup\{\operatorname{Re}(\lambda)\,:\,\lambda\in\sigma(A)\}$$
we denote the \emph{spectral bound} of $A$. We recall that for generators of
positive semigroups $R(\lambda,A)\geq 0$ holds for $\lambda>\sb(A)$.

\begin{remark}
In the context of Definition~\ref{pos},
for $\la>\sb(A_{-1})=\sb(A)$ it is easy to see that
$R(\lambda,A_{-1})\in \cL(E_{-1})$ is a positive operator (i.e., maps positive
elements to positive elements). In fact, it will follow from Proposition~\ref{cone}
that $R(\lambda,A_{-1})$ is also positive as an operator in $\cL(E_{-1},E)$, see Remark \ref{rem-pos} below.

Let $B\in\cL(E,E_{-1})$. If $B$ is positive, i.e.\ $Bf\ge 0$ for all $f\in 
E_+$, then $R(\lambda,A_{-1}) B$ is
positive as an operator in $\cL(E)$, for all $\la>\sb(A)$. Conversely, if there
exists a sequence $(\la_n)_{n\in\NN}$ in $(\sb(A),\infty)$ tending to
$\infty$, and such that $R(\lambda_n,A_{-1}) B\ge 0$ for all $n\in\NN$, then
$B\ge0$. Indeed, if $f\in E_+$, then $R(\lambda_n,A_{-1}) Bf\in E_+$ for all
$n\in\NN$, and the convergence $\la_nR(\lambda_n,A_{-1}) Bf \to Bf$ in $E_{-1}$
($n\to\infty$) implies $Bf\in\Emop$.

In Example~\ref{RBpositive} we will show that positivity of $R(\lambda,A_{-1})
B$ for only a single $\la>\sb(A_{-1})$ does not imply the positivity of $B$.
\end{remark}

Next we establish some basic properties of the ordering on $E_{-1}$.

\begin{proposition}\label{cone}
Let $E$ be a real Banach lattice and $(T(t))_{t\geq 0}$ a positive
$C_0$-semigroup on $E$. The set $\Emop$ is a closed convex cone in $E\mo$, 
satisfying
\begin{equation*} E_+=\Emop\cap E. \end{equation*}
\end{proposition}

\begin{proof}
Taking closures in the inclusions
$E_++E_+\sse E_+$ and $\alpha E_+\sse E_+$ for $\alpha\geq 0$, one obtains the
corresponding inclusions for $\Emop$. Also, $\Emop$ is closed as the
closure of $E_+$.
To show $\Emop\cap (-\Emop)=\{0\}$, let $f\in \Emop$
and assume also that $-f\in \Emop$. Then there exist sequences
$(f_n)_{n\in\NN},(g_n)_{n\in\NN}$ in $E_+$ such that $f_n\to f$ and $g_n\to -f$
in $E\mo$, and thereby $f_n+g_n\to
0$ in $E\mo$, as $n\to\infty$.
Choose $\la>\sb(A)$ and let the norm $\|\cdot\|_{-1}$ be defined in terms of
this $\la$. Note that
$0\leq f_n\leq f_n+g_n$, and hence $0\le R(\lambda,A)f_n\leq
R(\lambda,A)(f_n+g_n)$, by the positivity of the semigroup.
Therefore
\[
\|f_n\|_{-1} = \|R(\lambda,A)f_n\| \leq \|R(\lambda,A)(f_n+g_n)\|=\|f_n+g_n
\|_{-1}\to 0
\]
as $n\to\infty$. This shows that $f=0$.

Finally, to show that the definition of positivity in the extrapolation space
is compatible with the original ordering, we note that $E_+\subseteq \Emop\cap
E$ is immediate from the definition.

To prove the reverse inclusion let $f\in \Emop\cap E$. Then there exists a
sequence $(f_n)_{n\in\NN}$ in $E_+$ such that $\|f-f_n\|_{-1}\to0$
($n\to\infty$). Recalling that the norm $\|\cdot\|_{-1}$ can be defined using
any $\la \in\rho(A)$ we obtain
\begin{equation*}
\|R(\la,A)f - R(\la,A)f_n\| \to 0 \text{ as }n\to \infty,
\end{equation*}
for all $\la>\sb(A)$. Because of $R(\la,A)f_n\in E_+$ for all $n\in\NN$ this
implies that $R(\la,A)f\in E_+$ for all $\la>\sb(A)$. From $\la R(\la,A)f \to
f$ (in $E$) as $\la\to\infty$ we therefore obtain $f\in E_+$.
\end{proof}

\begin{remark}\label{rem-pos}
It is important to keep in mind the following simple consequence of the
properties shown in Proposition~\ref{cone}.

In the context of this proposition, let $C\colon E_{-1}\to E$ be an operator.
Then $C$ is positive if and only if $C$ is positive as an operator from
$E_{-1}$ to $E_{-1}$.
\end{remark}

\begin{remark}
The extrapolation space for a positive semigroup is not a Banach lattice, in
general. This will be shown by Examples \ref{application} and \ref{AMmitunit}.
\end{remark}

If the norm on a Banach lattice $E$ satisfies
\begin{equation}
\|\sup\{f,g\}\|=\sup\{\|f\|,\|g\|\}
\end{equation}
for all $f,g\in E_+$, then the Banach lattice $E$ is called an \emph{abstract
M-space} or an \emph{AM-space}. If the norm on a Banach lattice $E$ satisfies
\begin{equation}
\|f+g\|=\|f\|+\|g\|
\end{equation}
for all $f,g\in E_+$, then the Banach lattice $E$ is called an \emph{abstract L-space} or an \emph{AL-space}.

\section{Perturbations of Resolvent Positive Operators}\label{pert-res-pos}

In this section we provide a technical tool which will be needed in the proof
of the main result.

Let $E,F$ be ordered real Banach spaces.
Let $E_+$ be \emph{generating} ($E=E_+-E_+$) and \emph{normal}
($E' = E'_+-E'_+$), and assume that the norm on $\cL(E)$ is
\emph{monotone} ($A,B\in\cL(E)$, $0\le A\le B$ implies $\|A\|\le\|B\|$). (These
conditions on $E$ are satisfied if $E$ is a Banach lattice.)

\begin{lemma}\label{lemma-res-pos}
Let $Q\in\cL(E,F)$ be an isomorphism of Banach spaces, $Q\tmo\colon F\to E$
positive, and let $B\colon E\to F$ be a positive operator. (This implies
that $Q\tmo B\in\cL(E)$ by Batty and Robinson \cite[Proposition 1.7.2]{Bat-Rob}; hence
$B\in\cL(E,F)$.) Then the following conditions are equivalent:
\begin{iiv}
\item $\spr(Q\tmo B) < 1$,
\item $Q-B$ continuously invertible, with $(Q-B)\tmo\in\cL(F,E)$ positive.
\end{iiv}
If these properties are satisfied, then
\begin{equation*} 
(Q-B)\tmo = \biggl(\sum_{n=0}^\infty\bigl(Q\tmo B\bigr)^n\biggr)Q\tmo \ge Q\tmo.
\end{equation*}
\end{lemma}

\begin{proof}
First we show `(i)\timplies(ii)' and the additional assertion. Combining
condition~(i) and $Q\tmo B \ge 0$ with the Neumann series we obtain
\begin{equation*}
\bigl(I -Q\tmo B\bigr)\tmo = \sum_{n=0}^\infty\bigl(Q\tmo B\bigr)^n \ge I
\ge 0,
\end{equation*}
where $I$ denotes the identity operator in $E$.
Using the decomposition
\begin{equation*}
Q-B =Q\bigl(I - Q\tmo B)
\end{equation*}
one sees the  continuous invertibility of $Q-B$ and
\begin{equation*}
(Q-B)\tmo = \bigl(I - Q\tmo B)\tmo Q\tmo
= \biggl(\sum_{n=0}^\infty\bigl(Q\tmo B\bigr)^n\biggr)Q\tmo \ge Q\tmo.
\end{equation*}

For the proof of `(ii)\timplies(i)' we first note that the identity
\begin{equation*}
\sum_{j=0}^n\bigl(Q\tmo B)^j Q\tmo (Q-B) = I -\bigl(Q\tmo B)^{n+1}
\end{equation*}
implies
\begin{align*}
\sum_{j=1}^{n+1}\bigl(Q\tmo B\bigr)^j
&=\Bigl(I-\bigl(Q\tmo B\bigr)^{n+1}\Bigr)(Q-B)\tmo B\\
&=(Q-B)\tmo B -\bigl(Q\tmo B\bigr)^{n+1} (Q-B)\tmo B
\le (Q-B)\tmo B.
\end{align*}
Now the monotonicity of the norm in $\cL(E)$ implies
$(1,\infty)\sse\rho\bigl(Q\tmo B\bigr)$ and
\begin{equation*}
\sup_{\mu>1}\bigl\|\bigl(\mu -Q\tmo B\bigr)\tmo\bigr\|
\le\bigl\|I + \bigl(Q-B)\tmo B\bigr\|,
\end{equation*}
and this implies that $[1,\infty)\sse\rho(Q\tmo B)$. From the positivity
of $Q\tmo B$ we obtain $\spr(Q\tmo B)\in\sigma(Q\tmo B)$,
by Pringsheim's theorem (see Schaefer \cite[Appendix, 2.2]{Sch0}), and this finally
implies $\spr(Q\tmo B)< 1$.
\end{proof}


Now we specialise the hypotheses to the case that $E$§ is a real Banach
lattice, $(T(t))_{t\ge0}$ a positive $C_0$-semigroup on $E$, with generator
$A$, and that $F:=E_{-1}$ is the corresponding extrapolation space.
The following result is a version of Voigt \cite[Theorem~1.1]{Voi89}, with the
operator product $B(\la-A)\tmo$ replaced by $(\la-A\mo)\tmo B$.

\begin{theorem}\label{thm-res-pos}
Let $\la > \sb(A)$, and let $B\colon E\to E_{-1}$ be a positive operator.
(Recall from Lemma~\ref{lemma-res-pos} that then $(\la-A\mo)\tmo B\in\cL(E)$
and $B\in\cL(E,E_{-1})$.)
Then the following conditions are equivalent:
\begin{iiv}
\item $\spr\bigl((\la - A_{-1})\tmo B\bigr) < 1$,
\item $\la\in\rho(A_{-1} + B)$ and $(\la -A_{-1} - B)\tmo \ge 0$.
\end{iiv}
(In condition (ii) we consider $A_{-1}$ and $B$ as operators in
$E_{-1}$, with domain $E$.)

If these properties are satisfied, then
\begin{equation*}
(\la-A_{-1}-B)\tmo
= \biggl(\sum_{n=0}^\infty\bigl((\la-A_{-1})\tmo
B\bigr)^n\biggr)(\la-A_{-1})\tmo
\ge (\la - A_{-1})\tmo,
\end{equation*}
$\sb(A_{-1}+B) < \la$, and $(\mu - A_{-1}-B)\tmo \ge 0$ for all $\mu \ge \la$
(i.e., $A_{-1}+B$ is
resolvent positive, in the terminology of Arendt \cite{Are}).
\end{theorem}

\begin{proof}
With $Q:=\la- A_{-1}$ and in view of Remark~\ref{rem-pos}, all
the statements except for the last one are immediate consequences of
Lemma~\ref{lemma-res-pos}. However, from (i) one immediately obtains
$\spr\bigl((\mu-A_{-1})\tmo B\bigr) < 1$ for all $\mu\ge\la$, therefore
$[\la,\infty)\sse \rho(A_{-1}+B)$ and $(\mu - A_{-1}-B)\tmo\ge0$ for all
$\mu\ge\la$.
\end{proof}

\section{Perturbation Theory with Positive Operators}\label{perturbationtheory}

The cornerstone of the proof of our main result will be the Desch--Schappacher
perturbation theorem, which we cite here from Engel and Nagel \cite[Chapter III,
Corollaries 3.2 and 3.3]{EN}.

\begin{theorem}\label{EN}
Let $A$ be the generator of a $C_0$-semigroup $(T(t))_{t\geq 0}$
on a Banach space $X$ and let $B\in \cL(X,X_{-1})$. Moreover, assume that there
exist $\tau>0$ and $K\in [0,1)$ such that
\begin{iiv}
\item $\displaystyle\int_0^{\tau}T_{-1}(\tau-s)Bu(s)\dd s\in X$,
\item $\displaystyle \left\|\int_0^{\tau}T_{-1}(\tau-s)Bu(s)\dd s \right\|\leq K\|u\|_{\infty}$
\end{iiv}
for all continuous functions $u\in \rC([0,\tau],X)$. Then the operator
$(A_{-1}+B)\partof_X$
generates a
$C_0$-semigroup $(S(t))_{t\geq 0}$ on $X$. Furthermore this
semigroup is given by the Dyson--Phillips series
\begin{equation}\label{DP-1}
S(t)=\sum\limits_{n=0}^{\infty}S_n(t), \mbox{ for all }t\geq 0,
\end{equation}
where $S_0(t):=T(t)$ and
\begin{equation}\label{DP-2}
S_n(t)f:=\int_0^tT_{-1}(t-s)BS_{n-1}(s)f\dd s \mbox{ for all }f\in X.
\end{equation}
In this case $B$ is said to be a Desch--Schappacher perturbation of $A$.

\end{theorem}

Let us state and prove our result for AM-spaces and positive semigroups in a
special case, using the above theorem.

\begin{proposition}\label{prop-AM-Desch}
Let $E$ be a real AM-Space, $(T(t))_{t\geq 0}$ a
positive $C_0$-semigroup on $E$ with generator $A$. Let $B\in
\cL(E,E_{-1})$ be a positive operator and suppose further that there exists
$\la>\sb(A)$ such that
$K:=\|R(\la,A_{-1})B\|<1$. Then $(A_{-1}+B)\partof_E$
is the generator of a positive
$C_0$-semigroup $(S(t))_{t\geq 0}$, and the extrapolation space
$E\mo$ for this semigroup is the same as for $(T(t))_{t\ge0}$.
\end{proposition}

The key technical tool in verifying the conditions of Theorem~\ref{EN} will be
the following lemma, which we state separately.

\begin{lemma}\label{x-1}
Let $E$ be a real Banach lattice, $E_{-1}$ the extrapolation space for a
positive $C_0$-semigroup $(T(t))_{t\geq 0}$ on $E$, let $B\in
\cL(E,E_{-1})$ be a positive operator, and let $\tau>0$. Then we have:
\begin{iiv}
\item $(T_{-1}(t))_{t\geq 0}$ is positive.
\item For each step function
$u\rem{(\cdot)=\sum\limits_{n=0}^{N}u_n\chi_{I_n}(\cdot)}\in
\rL^{\infty}([0,\tau];E)$ we have
\begin{equation*} \int_0^{\tau}T_{-1}(s)Bu(s)\dd s\in E.\end{equation*}
\item For all $f\in E_+$ we have $\int_0^{\tau}T_{-1}(s)Bf\dd s
\in E_+$.
\item If in addition $(T(t))_{t\geq 0}$ is exponentially stable, then we have
\begin{equation*} \int_0^{\tau}T_{-1}(s)Bf\dd
s\leq\int_0^{\infty}T_{-1}(s)Bf\dd s \end{equation*}
in $E$, for all $f\in E_+$.
\end{iiv}
\end{lemma}

\begin{proof}
(i) This follows because $T_{-1}(t)$ is the continuous extension of
$T(t)$, for all $t\ge0$.

(ii) Let $u\in \rL^{\infty}([0,\tau];E)$ be a step function, i.e.,
$u(t)=\sum\limits_{n=1}^{N}u_n\chi_{I_n}(t)$ where $u_1,\dots,u_N\in E$,
$I_1,\dots,I_N\sse[0,\tau]$ are pairwise disjoint
intervals with $\bigcup_{n=1}^{N}I_n=[0,\tau]$, and where $\chi_{I_n}$ denotes 
the indicator function of $I_n$.
It suffices to show
\[ \int_{I_n}T_{-1}(s)Bu_n\dd s=\int_{t_{n-1}}^{t_n}T_{-1}(s)Bu_n\dd s\in E, \]
where $(t_{n-1},t_n)\sse I_n\sse [t_{n-1},t_n]$. With the substitution
$s'=s-t_{n-1}$ we get
\begin{align*} \int_{t_{n-1}}^{t_{n}}T_{-1}(s)Bu_n\dd s
&=\int_{0}^{t_{n}-t_{n-1}}T_{-1}(s+t_{n-1})Bu_n\dd s \\
&=T_{-1}(t_{n-1})\int_{0}^{t_{n}-t_{n-1}}T_{-1}(s)Bu_n\dd s.
\end{align*}
Because $(T_{-1}(t))_{t\geq 0}$ is a $C_0$-semigroup on $E_{-1}$
with generator $A_{-1}$, we have that
$\int_{0}^{t_{n}-t_{n-1}}T_{-1}(s)Bu_n\dd s$ belongs to $\dom(A_{-1})=E$, and
the assertion follows.

Statements (iii) and (iv) follow directly from Proposition \ref{cone}.
\end{proof}


\begin{proof}[Proof of Proposition \ref{prop-AM-Desch}]
In the first (main) part of the proof we will assume that the given semigroup
is exponentially stable, and that $\la =0$.

Let $\tau>0$. Let us denote by $\rT([0,\tau];E)$ the vector space of $E$-valued
step functions. In fact, $\rT([0,\tau];E)$ is a normed vector lattice, a
sublattice of $\rL^\infty([0,\tau];E)$. We define a linear operator
$R\colon\rT([0,\tau];E)\to E$ by
\[
Ru := \int_0^\tau T\mo(\tau-s)Bu(s)\di s.
\]
Note that Proposition \ref{cone} implies that $R$ is a positive operator. We
show that
\begin{equation}\label{equ-main-est}
\|Ru\|_E \le K\|u\|_\infty,
\end{equation}
for all $u\in\rT([0,\tau];E)$.
\sloppy

First, let
$u$ be a positive step function, $u=\sum_{n=1}^Nu_n\chi_{I_n}$ as above, with
$u_1,u_2,\dots,u_N\ge0$. Then $0\le u\le
z\chi_{[0,\tau]}$, where $z:=\sup_nu_n$.
We conclude, with the help of
Proposition \ref{cone} and Lemma \ref{x-1}, that
\begin{align*}
\left\|Ru\right\|&\leq\left\|\int_0^{\tau}T_{-1}(\tau-s)Bz\dd s \right\|\\
&\leq \left\|\int_0^{\infty}T_{-1}(\tau)Bz\dd s \right\|\leq
\|A^{-1}_{-1}B\|\|z\|\\
&=K\|z\|=K\|\sup\limits_nu_n\|=K\sup\limits_n\|u_n\|=K\|u\|_{\infty},
\end{align*}
where we have used the AM-property of $E$ in the last line. If $u$ is an
arbitrary
$E$-valued step function, then $u=u^+-u^-$, $|Ru|=|Ru^+-Ru^-|\le
Ru^++Ru^-=R|u|$, hence $\|Ru\|\le\|R|u|\|\le K\||u|\|_\infty =K\|u\|_\infty$.
\fussy

The estimate \eqref{equ-main-est} implies that $R$ possesses a (unique
linear) continuous extension -- still denoted by $R$ -- to the closure of
$\rT([0,\tau];E)$ in $\rL^\infty([0,\tau];E)$. This closure contains
$\rC([0,\tau];E)$, and the estimate \eqref{equ-main-est} carries over to all
$u$ in the closure.

If $u\in \rC([0,\tau];E)$, and $(u_n)_{n\in\NN}$ is a sequence in
$\rT([0,\tau];E)$ converging to $u$ uniformly on $[0,\tau]$, then $Ru_n\to Ru$
in $E$. But also
\[
Ru_n = \int_0^\tau T\mo(\tau-s)Bu_n(s)\di s\to
                           \int_0^\tau T\mo(\tau-s)u(s)\di s
\]
in $E\mo$, because $B\colon E\to E\mo$ is continuous and $(T\mo(t))_{t\ge0}$ is
bounded on $[0,\tau]$. This implies that $\int_0^\tau T\mo(\tau-s)Bu(s)\di s=
Ru\in E$, and that
\[
\left\|\int_0^\tau T\mo(\tau-s)Bu(s)\di s\right\|\le K\|u\|_\infty.
\]

Therefore both conditions in Theorem \ref{EN} are satisfied;
hence $(A_{-1}+B)\partof_E$ generates a $C_0$-semigroup $(S(t))_{t\geq 0}$
which is given by the Dyson--Phillips series (see Equations \eqref{DP-1} and
\eqref{DP-2}). Using Lemma~\ref{x-1}(i) and Remark~\ref{rem-pos} we conclude
that the iterates $S_n(t)$ as well as the semigroup operators $S(t)$ are
positive. This shows all the statements for the present case, except for the
assertion concerning the extrapolation spaces.

For the general case we note that from a basic rescaling procedure in semigroup
theory we know that $A$ is the generator of the (positive)
$C_0$-semigroup $(T(t))_{t\geq 0}$ if and only if $A-\lambda$ is the generator
of the (positive) $C_0$-semigroup $(\eul^{-\lambda t}T(t))_{t\geq 0}$. Observe
that the function $(\sb(A),\infty)\ni\la\mapsto \|R(\la,A_{-1})B\|$ is
decreasing.
Now choose
$\lambda>\sb(A)$ such that $A-\lambda$ generates a positive exponentially
stable $C_0$-semigroup and such that $\|R(\la,A_{-1})B\|<1$. Then the case
treated so far implies that
$(A_{-1}-\lambda + B)\partof_E$ generates a positive $C_0$-semigroup.

Now we show the equality of the extrapolation spaces. We choose
\[\la>\max\bigl\{\sb(A),\sb\bigl((A\mo+B)\partof_E\bigr)\bigr\}\] and such that
$\|R(\la,A_{-1}))B\|<1$. From the identity
\[
(\la-A\mo -B) = (\la-A\mo)\bigl(I-(\la-A\mo)\tmo B\bigr)
\]
we obtain
\[
(\la-A\mo -B)\tmo = \bigl(I-(\la-A\mo)\tmo B\bigr)\tmo(\la-A\mo)\tmo.
\]
Restricting this equality to $E$ we conclude that
\[
(\la-(A\mo -B)\partof_E)\tmo = \bigl(I-(\la-A\mo)\tmo B\bigr)\tmo(\la-A)\tmo.
\]
In view of the continuous invertibility of the first operator on the right hand
side, this equality shows that the $\|\cdot\|\mo$-norms corresponding to $(A\mo
-B)\partof_E$ and $A$ are equivalent on $E$, and therefore the completions are
the same.
\end{proof}

\begin{proof}[Proof of Theorem \ref{AMresult}] The following is an
adaptation of the proof given in Voigt \cite[Proof of Theorem 0.1]{Voi89}. It is
assumed that there exists $\la>\sb(A)$ such that $\spr\bigl((\la-A\mo)^{-1} B\bigr)
<1$, which by Theorem~\ref{thm-res-pos} is equivalent to the requirement
$\la\in\rho(A\mo+B)$, $(\la-A\mo - B)\tmo \ge0$. From Theorem~\ref{thm-res-pos}
we then conclude that $\la\in\rho(A\mo+sB)$,
\[
(\la-A\mo)\tmo \le (\la - A\mo - sB)\tmo \le (\la-A\mo-B)\tmo
\]
for all $s\in(0,1)$. We choose $n\in\NN$ such that $\|(\la-A\mo-B)\tmo B\|<n$.
This implies
$\|(\la-A\mo-(j/n)B)\tmo(1/n)B\|<1$
for all $j=0,\ldots,n-1$.

Applying Proposition~\ref{prop-AM-Desch} successively to the operators
\[
A,(A\mo+(1/n)B)\partof_E,\ldots,(A\mo+((n-1)/n)B))\partof_E,
\]
with the
perturbation $(1/n)B$, the desired result is obtained. An important point in
this sequence of steps is that the extrapolation space $E\mo$ does not change;
this issue is taken care of by the last statement of
Proposition~\ref{prop-AM-Desch}.
\end{proof}

\begin{remark}\label{rem-main-result}
(a) In our main theorem, Theorem~\ref{AMresult}, the hypothesis that `there
exists $\lambda>\sb(A)$ such that
$\spr\bigl((\lambda-A_{-1})^{-1}B\bigr)<1$' could have been formulated
equivalently as `there exists $\lambda>\sb(A)$ such that $\la\in\rho(A_{-1} +
B)$ and $(\la -A_{-1} - B)\tmo \ge 0$', or else as `$A_{-1}+B$ is
resolvent positive'. This is a consequence of Theorem~\ref{thm-res-pos}.

(b) Similarly, in Theorem~\ref{cor:miya}, as it is stated in Voigt
\cite[Theorem 0.1]{Voi89}, the condition that `there exists $\la > \sb(A)$ such
that $\spr(B(\la - A)\tmo) < 1$' appears as `$A+B$ is resolvent positive'. The
equivalence of these conditions is a consequence of Voigt \cite[Theorem 1.1]{Voi89}.
\end{remark}


\section{Examples}\label{examples}

We start this section with an application of Theorem \ref{AMresult}.

\begin{example}\label{application}
Let $h\in\rL^1(0,1)_+$.
Consider the partial differential equation

\begin{align*}
\frac{\partial}{\partial t}u(t,x)&=\frac{\partial}{\partial
x}u(t,x)+\int_0^1u(t,y)\di y\cdot h(x),\qquad
x\in[0,1], t\geq 0, \\
u(0,x)&=u_0(x),\quad u(t,1)=0, \qquad x\in[0,1], t\geq 0.
\end{align*}
Trying to interpret this equation as an
abstract Cauchy problem on the AM-space $E:=  \{f\in \rC([0,1]):f(1)=0 \}$ with
norm $\|f\|_{\infty}:=\sup\limits_{x\in[0,1]}|f(x)|$,
\begin{align*}
\dot u(t)&=Au(t)+Bu(t) \\
u(0)&=u_0,
\end{align*}
where the operator $A$ is defined by
\begin{equation}\label{ableitung}
Af=f', \qquad \dom(A)=\{f\in \rC^1[0,1]:\,f(1)=f'(1)=0 \},
\end{equation}
one realises that it is not evident how to associate the right
hand side of the equation with a linear operator in $E$.
In Engel and Nagel \cite[Chapter II, Example 3.19(i)]{EN}, it is shown that $A$ 
is the
generator of the nilpotent positive left-shift semigroup $(T(t))_{t\geq 0}$ with
$\sb(A)=-\infty$, given by
\begin{equation}\label{leftshift}
(T(t)f)(x)=\begin{cases}
f(s+t) & \text{if }x+t\leq 1, \\
0 & \mbox{otherwise}.
\end{cases}
\end{equation}

We want to calculate the extrapolation space of $E$ for the generator $A$. Our
aim is to show the equality
\begin{equation}\label{EinD}
E\mo = \{ g\in\mathscr{D}(0,1)':\, g=\partial f \mbox{ for
some }f\in E \},
\end{equation}
where $\mathscr{D}(0,1)=C_{\rm c}^\infty(0,1)$ denotes the usual space of
`test functions', with the inductive limit topology, $\mathscr{D}(0,1)'$ its
dual space, and $\partial$ is differentiation on distributions.

First we note that the `standard embedding' $j\colon
E\hookrightarrow\mathscr{D}(0,1)'$ can be extended to a mapping
\[
j\mo\colon E\mo \to  \mathscr{D}(0,1)',
\]
defined by
\[
\langle j\mo(g),\phi\rangle:=\left\langle
A^{-1}_{-1}g,-\phi'\right\rangle
=-\int_0^1\bigl(A^{-1}_{-1}g\bigr)(x)\phi'(x)\di x.
\]
Indeed, if $g\in E$, then
\[
\langle j\mo(g),\phi\rangle = - \int_0^1\bigl(A\tmo g)(x)\phi'(x)\di x
=\int_0^1\bigl(A\tmo g)'(x)\phi(x)\di x = \int_0^1 g(x)\phi(x)\di x,
\]
which shows that $j\mo$ is an extension of $j$. In fact, the definition of
$j\mo$ shows that $j\mo = \partial\circ j\circ A\mo\tmo$, and this formula shows
that $j\mo$ maps $E\mo$ continuously to $\mathscr D(0,1)'$. Finally we note that
$j\mo$ is injective. Indeed, if $g\in E\mo$ is such that $j\mo(g)=0$, then
$\int_0^1\bigl(A\mo\tmo g\bigr)(x)\phi'(x)\di x =0$ for all $\phi\in\mathscr
D(0,1)$, which implies that the continuous function $A\mo\tmo g$ is constant,
and this constant is zero because $(A\mo\tmo g\bigr)(1) =0$. The injectivity of
$A\mo\tmo$ then implies $g=0$.

Rewriting the above formula for $j\mo$ as
\[
j\mo\circ A\mo =\partial\circ j,
\]
valid on $E$, we see the validity of \eqref{EinD} as well as the property that
in the image of $E$ in $\mathscr D(0,1)'$ the operator $A\mo$ acts as
differentiation $\partial$.

Next we are going to show that, in the image of $E\mo$ in $\mathscr D(0,1)'$,
one has
\begin{align}\label{equEmo}
\begin{split}
\Emop &= \{ \mu :\, \mu\text{ a finite continuous positive Borel measure on
}(0,1) \}\\
&= \{ g\in\mathscr{D}(0,1)':\, g=\partial f \mbox{ for
some increasing function }f\in E \},
\end{split}
\end{align}
where the second equality is standard and will not be discussed further.

So, let $g\in\Emop$. Then there exists a sequence $(g_n)_{n\in\NN}$ in $E_+$
such that $g_n\to g$ in $E\mo$, and thus in $\mathscr D(0,1)'$. Hence
\[
\int_0^1g_n(x)\phi(x)\di x \to \langle g,\phi\rangle\qquad(n\to\infty),
\]
for all $\phi\in\mathscr D(0,1)$, and therefore $\langle g,\phi\rangle\ge0$ for
all $0\le\phi\in\mathscr D(0,1)$, i.e., $g$ is a `positive distribution'.
It is known that this implies that $g$ is a positive Borel measure;
see Schwartz \cite[Chap.\,I, Th\'eor\`eme V]{Schw}. As $g$ is also the distributional
derivative of a function $f\in E$, it follows that $f$ is increasing and that
$\mu$ is finite and continuous (i.e., does not have a discrete part).

For the reverse inclusion, let $f\in E$ be an increasing function. Then one
shows by standard methods of Analysis that $f$ can be approximated in $E$ by
a sequence $(f_n)_{n\in\NN}$ in $E\cap \rC^1[0,1]$, all $f_n$ increasing and
vanishing in a neighbourhood of $1$. This implies that $f_n'\in E_+$ for all
$n\in\NN$, and that $f_n'\to A\mo f=\partial f$ ($n\to\infty$) in $E\mo$; hence
$g:=\partial f\in\Emop$.

As a consequence of \eqref{equEmo} we obtain that $\Emop-\Emop$, the set of
finite continuous signed measures, is a proper subset of $E\mo$ (because the
distributional derivative of a function in $E$ that is not of bounded variation
is not a measure). This implies that $\Emop$ is not generating in $E\mo$, and
therefore $E\mo$ is not a lattice -- let alone a Banach lattice.

\medskip
Now we come back to treating the initial value problem stated at the beginning.
There exists an increasing function $f\in E$ such that $h = \partial f$, and
therefore $h\in\Emop$. Hence the operator
%
$B\in \cL(E,E_{-1})$, defined by $g\mapsto\int_0^1g(x)\dd
x\cdot h$, is a positive operator.
We calculate
\begin{align*}
\|(\la-A_{-1})\tmo B\|
&=\sup_{\|g\|=1}\|(\la-A_{-1})^{-1}Bg\|_{\infty}\\
&=\sup_{\|g\|=1}\left\|\left(\int_0^1g(x)\dd x\right)\cdot
(\la-A\mo)^{-1}h\right\|_{\infty}\\
&= \|(\la-A\mo)\tmo h\|_E.
\end{align*}
It is a standard fact from semigroup theory that $\|(\la-A\mo)\tmo h\|_E\to 0$
as $\la\to\infty$, for all $h\in E\mo$; hence
$\|(\la-A\mo)\tmo B\|<1$ for large $\la$. Therefore Theorem \ref{AMresult}
implies that $(A_{-1}+B)\partof_E$ is
the generator of a positive semigroup.
\end{example}

\begin{remark}
(a) Clearly, the function $h$ in Example \ref{application} could also be
replaced by measures in $\Emop$ which are not absolutely continuous
with respect to the Lebegue measure.

(b) If, more strongly, $h\in\rL_\infty(0,1)$ (but not necessarily
positive), then $h$ belongs to the extrapolated Favard class
$F_0=\rL_\infty(0,1)$ of the semigroup $(T(t))_{t\ge0}$, and the generator
property of $(A\mo+B)\partof_E$ follows from Engel and Nagel \cite[Chapter III,
Corollary 3.6]{EN}.
\end{remark}

Example \ref{application} shows that $E_{-1}$ need not be a Banach lattice if
$E$ is an
AM-space without order unit. Next we present a counterexample of an AM-Space
with order unit.

\begin{example}\label{AMmitunit}
Consider the space $E=\{f\in \rC[0,1]:f(0)=f(1) \}$ and the operator
\begin{equation*}
Ah=h', \qquad \dom(A)=\{f\in \rC^1[0,1]\cap E:\,f'(0)=f'(1) \}.
\end{equation*}
It is shown in Nagel \cite[Chapter A-I, page 11]{N}, that $A$ is the generator of the positive periodic bounded semigroup $(T(t))_{t\geq 0}$ 
given by
\begin{equation*}
(T(t)f)(s)=f(y) \qquad \mbox{ for }y\in[0,1], y=s+t \mbox{ mod }1.
\end{equation*}

The description of $E\mo$ and $\Emop$ will be analogous to the description in
Example~\ref{application}, but slightly more involved because in the present
case the operator $A$ is not invertible. Our first aim is to obtain the
representation
\begin{equation}\label{equEmo-per}
E\mo = \{ h\in\mathscr D(0,1)':\, h=f-\partial f\text{ for some }f\in E\}.
\end{equation}
We refer to Engel and Nagel \cite[Chapter II, Example~5.8(ii)]{EN} for this
expression in a
similar context.
As above, $j\colon E\to \mathscr
D(0,1)'$ will be the standard injection. We define its extension to $E\mo$
by
\[
\langle j\mo(h),\phi\rangle :=\int_0^1\bigl((1-A\mo)\tmo
h\bigr)(x)(\phi+\phi')(x)\di x.
\]
Then $j\mo=(1-\partial)\circ j\circ(1-A\mo)\tmo$. To
show the injectivity of $j\mo$, assume that $h\in E\mo$ is such that
$j\mo(h)=0$. Then $f:=(1-A\mo)\tmo h\in E$ satisfies
\[
0=\int_0^1 f(x)(\phi+\phi')(x)\di
x=\int_0^1f(x)\frac{1}{\exp(x)}(\exp\cdot\phi)'(x)\di x
\]
for all $\phi\in\mathscr D(0,1)$. This implies that $f\frac{1}{\exp}$ is
constant, and therefore $f(0)=f(1)$ shows that $f=0$, $h=0$. Rewriting the
formula for $j\mo$ as
\[
j\mo\circ(1-A\mo) = (1-\partial)\circ j
\]
we obtain \eqref{equEmo-per}. For the following, it will be important to keep
in mind the identity $f-\partial f = \exp\partial\bigl(\frac{f}{\exp}\bigr)$,
for $f\in E$.

For the present context we are going to show that
\begin{equation}\label{equEmop-per}
\Emop = \{\mu:\, \mu \text{ a finite continuous positive Borel measure on
}(0,1)\}.
\end{equation}
The inclusion `$\sse$' is shown as in Example~\ref{application}. For the
reverse inclusion let $\mu$ be as on the right hand side of
\eqref{equEmop-per}. Then there exists $g\in \rC[0,1]$ with $\frac{g}{\exp}$
increasing and such that
$\frac{1}{\exp}\mu=\partial\bigl(\frac{g}{\exp}\bigr)$. Note that then also
\[
\frac{1}{\exp}\mu=\partial\bigl(\frac{1}{\exp}(g + c\exp)\bigr)
\]
for all $c\in\RR$. It is clear that there exists a unique $c\in\RR$ such that
$f:=g+c\exp\in E$, and this implies
$\mu=\exp\partial\bigl(\frac{f}{\exp}\bigr)= f-\partial f$.
In order to show that $\mu\in\Emop$ we still have to approximate $f$ suitably.
To do so we first extend $f$ to $\RR$ as a continuous periodic function. Then
we define $f_k:=\rho_k*f$, where $(\rho_k)_{k\in\NN}$ is a $\delta$-sequence in
$C_{\rm c}^\infty(\RR)$. Then $f_k\in \rC^\infty(\RR)\cap E$ for all $k\in\NN$,
$f_k\to f$ in $E$ as $k\to\infty$, and it is not too difficult to show that
$\frac{f_k}{\exp}$
is increasing for all $k$. Then $f_k -f_k'\in E_+$ for all $k\in\NN$, and
$f_k-f_k'\to f-\partial f$ ($k\to\infty$) in $E\mo$ implies
$\mu=f-\partial f\in\Emop$.

From \eqref{equEmo-per} and \eqref{equEmop-per} it follows as in
Example~\ref{application} that $E\mo$ is not a lattice and a fortiori not a
Banach lattice.
\end{example}

Our final example shows that, for an operator $B\in \cL(E,E_{-1})$ to be
positive it is not sufficient that $R(\lambda,A_{-1})B$ is positive in $\cL(E)$
for some $\la>\sb(A)$.

\begin{example}\label{RBpositive}
Let $E$ and $A$ be as in Example~\ref{application}. Define
\[
h:=-\chi_{[0,1/2)}+ \chi_{[1/2,1]}.
\]
Then the description of $\Emop$ in Example~\ref{application} shows that $h$ is
not positive in $E\mo$, because $h=\partial g$, for the function $g\in E$ given
by
\begin{equation*}
g(x)=\begin{cases}  -x & \text{if }x\in [0,\frac 12[, \\ x-1 &
\text{if }x\in[\frac 12,1], \end{cases}
\end{equation*}
which is not increasing. However, $(-A\mo)\tmo h = -g$ belongs to $E_+$.

Defining the operator $B\in\cL(E,E\mo)$ by
\[
Bf:=\int_0^1f(x)\di x\cdot h
\]
we see that $(0-A\mo)\tmo B\in\cL(E)$ is positive, but $B$ is not positive.
\end{example}

\section*{Acknowledgement}
The authors are grateful to B\'alint Farkas, Sven-Ake Wegner  and Hans Zwart for fruitful
discussions and helpful comments. The first author gratefully acknowledges
financial support by  the German Academic Exchange Service (DAAD).

\bibliographystyle{amsplain}

\begin{thebibliography}{10}

\bibitem{Are}
{W.~Arendt}, \emph{Resolvent positive operators}, Proc.~London Math.\ Soc.
\textbf{54} (1987), 321--349.

\bibitem{BKR}
{A. B\'atkai}, {M. Kramar Fijav\v{z}}, and A.~Rhandi, \emph{Positive Operator
  Semigroups: from Finite to Infinite Dimensions}, {Birkh\"auser}-Verlag,
  Basel, 2016.

\bibitem{Bat-Rob}
C.J.K.~Batty and D.W.~Robinson, \emph{Positive one-parameter semigroups on
ordered Banach spaces}, Acta Appl.\ Math.\ \textbf{2} (1984), 221--296.

\bibitem{BA}
J.~Banasiak and L.~Arlotti, \emph{Perturbation of Positive Semigroups with
  Applications}, Springer-Verlag, London, 2006.

\bibitem{AB}
O.~Burkinshaw C.D.~Aliprantis, \emph{Positive Operators}, Academic Press,
  Orlando, 1985.

\bibitem{Davies}
E.~B. Davies, \emph{One-Parameter Semigroups}, Academic Press, New York, 1980.

\bibitem{D}
W.~Desch, \emph{Perturbations of positive semigroups in {AL}-spaces}, preprint 1988.

\bibitem{EN}
K.-J. Engel and R.~Nagel, \emph{One-Parameter Semigroups for Linear Evolution
  Equations}, Springer-Verlag, New York, 2000.

\bibitem{Gol:85}
J.A. Goldstein, \emph{Semigroups of {O}perators and {A}pplications}, Oxford
University Press, New York, 1985.

\bibitem{MR2952349}
B. Jacob and H.~J. Zwart, \emph{Linear Port-{H}amiltonian Systems on
  Infinite-Dimensional Spaces},  Birkh\"auser-Verlag, Basel, 2012.

\bibitem{Kat:80}
T.~Kato, \emph{Perturbation Theory for Linear Operators}, Springer-Verlag, New York, 1980.

\bibitem{N}
R.~Nagel (ed.), \emph{One-Parameter Semigroups of Positive Operators}, Springer-Verlag,
  Berlin, 1986.

\bibitem{Pazy}
A.~Pazy, \emph{Semigroups of Linear Operators and Applications to Partial
  Differential Equations}, Springer-Verlag, New York, 1983.

\bibitem{Sch0}
H.H. Schaefer, \emph{Topological Vector Spaces}, Springer-Verlag,
New York, 1980.

\bibitem{Sch}
H.H. Schaefer, \emph{Banach Lattices and Positive Operators}, Springer-Verlag,
  Berlin, 1974.

\bibitem{Schw}
L.~Schwartz, \emph{Th\'eorie des distributions}, Hermann, Paris, 1966.

\bibitem{TW}
M.~Tucsnak and G.~Weiss, \emph{Observation and Control for Operator  Semigroups}, Birkh\"auser-Verlag, Basel, 2009.

\bibitem{Voi89}
J.~Voigt, \emph{On resolvent positive operators and positive {$C_0$}-semigroups
  on AL-spaces}, Semigroup Forum \textbf{38} (1989), 263--266.

\end{thebibliography}

\end{document}